\documentclass[12pt]{amsart} 

\usepackage[margin=1in]{geometry}
\usepackage{graphicx} 
\usepackage{amsmath}
\usepackage{amsfonts, amsthm}
\usepackage{amssymb}
\usepackage{enumerate}
\usepackage{tikz}
\usepackage[font=small,labelfont=bf]{caption}
\usepackage{hyperref} 
\usepackage{cleveref}

\newtheorem{theorem}{Theorem}[section]

\newtheorem{proposition}[theorem]{Proposition}
\newtheorem{lemma}[theorem]{Lemma}
\newtheorem{corollary}[theorem]{Corollary}

\theoremstyle{remark}
\newtheorem{remark}{Remark}[theorem]
\numberwithin{equation}{section}

\newcommand{\N}{\mathbb{N}}

\newcommand{\R}{\mathbb{R}}
\newcommand{\C}{\mathbb{C}}
\newcommand{\ld}[1]{\frac{#1'}{#1}} 
\renewcommand{\epsilon}{\varepsilon} 
\renewcommand{\leq}{\leqslant}
\renewcommand{\le}{\leqslant}
\renewcommand{\geq}{\geqslant}
\renewcommand{\ge}{\geqslant}

\title{Numerical Computations concerning Landau--Siegel Zeros}

\author{Rick F. Lu}
\author{Asif Zaman}
\author{Haonan Zhao}

\address{Department of Mathematics, University of Toronto, 40 St. George Street, Room 6290, Toronto ON, M5S 2E4, Canada}
\email{mr.ricklu@hotmail.com}
\email{asif.zaman@utoronto.ca}
\email{zhaohaonan1228@gmail.com}

\begin{document}

\begin{abstract} We computationally verify that if  $L(s,\chi)$ is a quadratic Dirichlet $L$-function modulo $q \leq 10^{10}$ then $L(\sigma,\chi) \neq 0$ for real $\sigma \ge 1-1/(5\log q)$. The number of verified moduli exceeds benchmarks due to Watkins (2004), Platt (2016), and Languasco (2023) by a factor between $66$ and $25$,000. Our new algorithm draws from zero-free region arguments.
\end{abstract}

\maketitle

\section{Introduction}
The standard zero-free region for  Dirichlet $L$-functions   states that there exists an absolute constant $c>0$ such that for all integers $q \ge 3$, the function $\prod_{\chi \pmod{q}} L(s,\chi)$ has at most one zero, counted with multiplicity, in the region given by 
    \begin{equation}
    \label{eqn:zfr}
       \sigma \ge 1-\frac{c}{\log \max( q, q|t|, 10)},
    \end{equation}
    where $s = \sigma+it \in \C$. Furthermore, if such a zero $\beta_1$ exists, then $\beta_1$ is real and its associated character $\chi$ is quadratic. This classical result was first established by Landau \cite{Landau_1918}, and  McCurley \cite{McCurley_1984} gave an explicit value of $c=1/10$ (cf. improvements by Kadiri \cite{kadiri} for $q \leq 4 \times 10^5$, Ahn and Kwon \cite{AhnKwon2014} for $|t| \leq \frac{1}{2\log q}$, and Louboutin \cite{Louboutin2015-1} for $t=0$).  
    
    The zero $\beta_1$ is commonly referred to as an \textit{exceptional zero} or  \textit{Landau--Siegel zero} due to celebrated works of Landau  \cite{Landau1935} and Siegel \cite{Siegel1935}. These exceptional zeros conjecturally do not exist, which is implied by the Generalized Riemann Hypothesis (GRH). This conjecture has remained open for over a century. It is notoriously related to fundamental topics in number theory, such as the size of class numbers \cite{Landau1935,Siegel1935},  primes in arithmetic progressions \cite{HeathBrown1990,FriedlanderIwaniec2003}, primes in short intervals \cite{FriedlanderIwaniec2004}, the infinitude of twin primes \cite{HeathBrown1983,FriedlanderIwaniec2019}, famous conjectures of Chowla, Hardy--Littlewood, and Goldbach \cite{FGIS2022, TaoTeravainen2022,MatomakiMerkoski2023}, primes represented by binary forms \cite{FriedlanderIwaniec2005,Zaman2016}, discriminants of elliptic curves \cite{FriedlanderIwaniec2005}, arithmetic quantum unique ergodicity \cite{Thorner2022}, non-vanishing of central values of $L$-functions \cite{IwaniecSarnak2000,BuiPrattZaharescu2021,CechMatomaki2024}, and the spacing of zeros of $L$-functions \cite{ConreyIwaniec2002,Baluyot2016}. While we do not aim to provide a detailed survey of this vast literature (see, e.g. \cite{Iwaniec2006,FriedlanderIwaniec2017}), we hope this sample illustrates the  significance of Landau--Siegel zeros. 

For a fixed modulus $q$ and fixed choice of constant $c$ in \eqref{eqn:zfr}, our main goal is to provide a new method which computationally verifies that this exceptional zero does not exist. 

\begin{theorem}
\label[theorem]{thm:main}
    If $q \le 10^{10}$ and $\chi$ is a quadratic Dirichlet character modulo $q$, then 
    \[
    L(\sigma,\chi) \neq 0 \quad \text{ for } \sigma \ge 1 - \frac{1}{5 \log q}. 
    \]
\end{theorem}

\noindent
Combined with McCurley's zero-free region, we obtain the following corollary. 

\begin{corollary} If $q \le 10^{10}$ then $\displaystyle\prod_{\chi \pmod{q}} L(s,\chi)$  has no zeros in region \eqref{eqn:zfr} with $c=1/10$.
\end{corollary}

We review comparable explicit results for quadratic Dirichlet characters $\chi$ modulo $q$. First, Bordignon \cite{Bordignon-ExplicitSZ,Bordignon-ExplicitSZ-II} showed that $L(\sigma,\chi) \neq 0$ for all real $\sigma \ge 1 - 80/(\sqrt{q} \log^2 q)$ and all moduli $q \ge 3$. \Cref{thm:main} provides a wider zero-free interval on the real axis for $q \le 10^{10}$.   Second, by improving on Rumely's \cite{Rumely1993} verification of GRH up to a fixed height and fixed modulus for all Dirichlet $L$-functions, the work of Platt \cite{Platt_2016} implies that $L(\sigma,\chi) \neq 0$ for $\sigma > 0$ and $q \le 4 \times 10^5$. Third, building on ideas of Low \cite{Low1968} with Epstein zeta functions for positive definite binary quadratic forms, Watkins \cite{Watkins_2003} verified that $L(\sigma,\chi) \neq 0$ for $\sigma > 0$ and $q \le 3 \times 10^8$ when $\chi$ is odd. Fourth, by relating $L(1,\chi)$ to $1-\beta_1$ via elementary means and computing $L(1,\chi)$ via a Gauss sum relation, Languasco \cite{Languasco_2023} showed that $L(\sigma,\chi) \neq 0$ for all real $\sigma \ge 1-0.0091/\log q$ and prime moduli $q \le 10^7$. 

\Cref{thm:main} improves upon the latter three results by a factor between $66$ and $25$,000 when counting the number of moduli $q$ for which exceptional zeros are eliminated  (though Platt and Watkins establish a much wider zero-free interval on the real axis). Our gains are partly through advances in modern computing hardware, and partly through the introduction of a new computational method inspired by classical zero-free region arguments. The basic idea is twofold: (i)  establish an explicit inequality involving the logarithmic derivative $-\frac{L'}{L}(\sigma,\chi)$ for $\sigma > 1$ assuming the existence of an exceptional zero; and (ii) show this inequality is violated after a numerical verification. The explicit inequality is derived from versions of Jensen's inequality and convexity bounds originating from Heath-Brown \cite{Heath-Brown_1992,HeathBrown2009} and made explicit by Thorner and Zaman \cite{Thorner_Zaman_2024}. See \S\ref{sect:Method} for an overview of our method.  

 To be precise, our algorithm addresses the slightly more general problem: given an absolute constant $c > 0$ and a quadratic character $\chi$ modulo $q$, we wish to verify that 
\begin{equation} \label{eqn:runtime-problem}
L(\sigma,\chi) \neq 0 \quad \text{ for } \sigma \ge 1 - \frac{c}{\log q}. 
\end{equation}
Our algorithm requires the evaluation of sums over primes similar to $\sum_{p \le x} \chi(p) \log p$, whose known asymptotic behaviour is entangled with the possible existence of Landau--Siegel zeros. Thus, it seems to us that we cannot unconditionally show our algorithm terminates without eliminating these exceptional zeros. Instead, by assuming GRH to estimate such sums, we conditionally show in \S\ref{sect:Method} that our algorithm verifies \eqref{eqn:runtime-problem}  with a runtime of $O(q^{0.444})$ provided $c=1/5$ as in Theorem \ref{thm:main} and $q$ is sufficiently large. This runtime improves to $O_{\epsilon}(q^{1/4+\epsilon})$ for $c \le c(\epsilon)$ sufficiently small and $q \geq q(c)$ is sufficiently large. On the other hand,  all three  previous methods require at least $O(q^{1/2})$ operations to verify \eqref{eqn:runtime-problem}. Indeed, Platt's method officially takes $O_{\epsilon}(q^{1+\epsilon})$ \cite[Lemma 6.1]{Platt_2016} since he verifies GRH for all Dirichlet $L$-functions modulo $q$, but one can speed this up to $O_{\epsilon}(q^{1/2+\epsilon})$ for a single $L$-function via standard techniques with the Riemann--Siegel formula. Watkins's method is restricted to odd $\chi$ and must compute all reduced positive definite binary quadratic forms of discriminant $-q < 0$ to evaluate  Epstein zeta functions, which requires $O_{\epsilon}(q^{1/2+\epsilon})$ operations. Languasco \cite[Section 3]{Languasco_2023} uses a classical identity involving $L(1,\chi)$, the Gauss sum of $\chi$, and a certain sum over character values $\chi(a)$ for $1 \le a\le q-1$. The last sum requires at least $O(q)$ computations. Our algorithm's improved efficiency bears out in practice on average over $q \le Q$ with $Q$ large, but for ``worst case'' moduli $q$, lower order terms in our explicit inequality have a noticeable impact as they decay on the scale of $1/\log q$. See \S \ref{sect:runtime} for runtime in theory and \S \ref{sect:proof-main}  for runtime in practice.

We also obtain a simple consequence for the special value $L(1,\chi)$. A lemma of Hoffstein \cite{Hoffstein_1980} allows us to deduce a uniform lower bound for it. 
\begin{corollary}\label[corollary]{cor:L1chi}
If $q \le 10^{10}$ and $\chi$ is a primitive quadratic character modulo $q$, then 
\[
L(1,\chi) \ge \frac{1}{8 \log q}. 
\]
\end{corollary}
\noindent
It is worthwhile to compare this corollary with three related results. First, Siegel \cite{Siegel1935} famously showed the ineffective bound $L(1,\chi) \gg_{\epsilon} q^{-\epsilon}$ for $\epsilon > 0$ and all primitive quadratic characters $\chi$ modulo $q$.  Second, Bennett, Martin, O'Bryant, and Rechnitzer \cite{BennettMartinOBryantRechnitzer-ExplicitPNAP} showed that $L(1,\chi) > 12/\sqrt{q}$ for all primitive quadratic characters $\chi$ modulo $q > 6677$.  Third, Mossinghoff, Starichkova, and Trudgian \cite{MossinghoffStarichkovaTrudgian2022} refined a method of Louboutin \cite{Louboutin2015} to show that  $|L(1,\psi)| \ge 1/(9.7 \log q)$ for all primitive \textit{non-quadratic} characters $\psi$ modulo $q$. Observe that \Cref{cor:L1chi} is stronger than the first two results for $q \leq 10^{10}$, and exceeds the strength of current  explicit bounds for non-quadratic characters.

Note  \Cref{cor:L1chi} can certainly be improved because class numbers and regulators of quadratic fields with discriminant $\Delta$ satisfying $q := |\Delta| \le 10^{11}$ have been unconditionally computed by Jacobson, Ramachandran, and Williams \cite{Jacobson_2006} for the imaginary case,  and recently by Bian, Booker, Docherty, Jacobson, and Seymour-Howell \cite{BBDJS2024} for the real case. By the class number formula, their data will surely lead to much stronger bounds than \Cref{cor:L1chi} of the shape $L(1,\chi) \ge b (\log\log q)^{-1}$ for some constant $b$. In fact, this stronger bound was numerically verified in the imaginary  case \cite[Section 3.1]{Jacobson_2006} with $b = 1.59$ and $q \le 10^{11}$.  We only record \Cref{cor:L1chi} to give a simple bound in both cases. 

On the other hand, this existing data for $L(1,\chi)$ does not appear to imply \Cref{thm:main} because estimates relating $L(1,\chi)$ and $1-\beta_1$ are not sufficiently strong. Assuming the exceptional zero $\beta_1 > 1- (10 \log q)^{-1}$ exists and $L(\beta_1,\chi) = 0$, the best unconditional estimate (see, e.g., Friedlander and Iwaniec  \cite{FriedlanderIwaniec-2018}) states that 
\[
C_1 \le \frac{L(1,\chi)}{1-\beta_1} \le C_2 (\log q)^2. 
\]
for some absolute constants $C_1, C_2 > 0$ and $q \ge 3$. The value $C_2 = 0.18$ is admissible \cite[Lemma 2.9]{BGTZ2024}, so the above essentially implies that 
\[
L(\sigma,\chi) \neq 0 \quad \text{ for } \sigma \ge 1 - \frac{L(1,\chi)}{0.18 (\log q)^2}. 
\]
Assuming GRH, Littlewood  \cite{Littlewood1928} showed that  $(\log\log q)^{-1} \ll L(1,\chi) \ll \log\log q$,  so it seems difficult to deduce \Cref{thm:main} using computed values of  $L(1,\chi)$. 

\subsection*{Organization} 
\Cref{sect:Method} gives an overview of our algorithm including a statement of our main numerical result (\Cref{thm:explicitvalues}), and an asymptotic runtime analysis of our algorithm. \Cref{sect:proof-main} contains the proofs of \Cref{thm:main} and \Cref{cor:L1chi} as well as details of our computations. \Cref{sect:preliminary} contains preliminary lemmas on logarithmic derivatives. \Cref{sect:convexity} contains explicit convexity bounds for $L$-functions along vertical lines. \Cref{sect:bound} assembles \Cref{thm:finalbound}, a technical form of \Cref{thm:explicitvalues}. \Cref{sect:optimizeR} contains the proof of \Cref{thm:explicitvalues} and heuristically justifies the choice of some parameters.

\subsection*{Code availability} The code associated with this paper is available on GitHub at 
\begin{center}
    \href{https://github.com/asif-z/landau-siegel-zero-tester/}{\tt https://github.com/asif-z/landau-siegel-zero-tester/}
\end{center}

\section{Method Overview} \label{sect:Method}

In this section, we give an outline of the algorithm and provide a runtime analysis. 

\subsection{Algorithm outline}
\label{sect:algorithm}
Our algorithm rests entirely on the following explicit inequality. It is the main theoretical content, and will be established in \S\ref{sect:optimizeR}. 
\begin{theorem}
\label[theorem]{thm:explicitvalues} Let $\chi$ be a primitive non-principal character modulo $q$. Assume $L(s,\chi)$ has a real zero at $\beta_1\in(0,1)$. Write $\sigma=1+r$ for $r>0$. Then 
    \begin{equation}
   \sum_{n=1}^\infty \Lambda(n)\frac{1+  \Re\{ \chi(n) \}}{n^\sigma} \le \frac{1}{r}-\frac{1}{\sigma-\beta_1}+\frac{\sigma-\beta_1}{R^2}+\phi\log q+E,
    \label{eqn:ineq}
    \end{equation}
    where $R=r+7/8$ and the values $r$, $\phi$, and $E$ are given by \Cref{table:ineq}.
\end{theorem}
\begin{remark} \label[remark]{rem:choice-of-constants}
The values of $R,\phi,$ and $E$ are determined by \Cref{thm:combineZetaL} where they implicitly depend on $r$. The four choices of $r$ are optimized for computational purposes. Note the pivotal constant $\phi$ satisfies $\phi \rightarrow 1/4^-$ as $r\rightarrow 0^+$, and $E=O(\log(1+r^{-1}))$. See \Cref{rem:suboptimalBounds}  and \S\ref{sect:optimizeR} for more details.
\end{remark}

\begin{table}[h!]
    \centering
    \begin{tabular}{c|c|c}
        $r$ & $\phi$ & $E$ \\
        \hline
        $1.6/10\log 10$ & $0.22675$ & $1.1614$ \\
        \hline
        $1.3/10\log 10$ & $0.23083$ & $1.1805$ \\
        \hline
        $1.2/10\log 10$ & $0.23222$ & $1.1870$ \\
        \hline
        $1.1/10\log 10$ & $0.23362$ & $1.1936$ \\
    \end{tabular}
    \caption{\label{table:ineq}Constants in \Cref{thm:explicitvalues}.}
\end{table}

Assuming \Cref{thm:explicitvalues}, we describe the basic idea of our algorithm. Let $\chi$ be a primitive quadratic character modulo $q$. By non-negativity and the fact that $\chi(n) \in \{-1,0,1\}$, it follows for any $\sigma > 1$ that 
\begin{equation} \label{eqn:non-negativity}
\sum_{n \in \mathcal{S}}\Lambda(n)\frac{1+ \Re\{\chi(n)\}}{n^\sigma} = \sum_{n \in \mathcal{S}}\Lambda(n)\frac{1+ \chi(n)}{n^\sigma}
\le\sum_{n=1}^\infty \Lambda(n)\frac{1+\chi(n)}{n^\sigma}    
\end{equation}
for any subset $\mathcal{S} \subset \N$. For a parameter $N \ge 2$, choose $\mathcal{S} = \{p^k:p\le N, p \text{ prime}, k \in \N\}$, in which case the left-hand expression becomes
\begin{equation} \label{eqn:primesum}
   \sum_{n \in \mathcal{S}}\Lambda(n)\frac{1+ \chi(n)}{n^\sigma} = \sum_{p\le N}(\log p)\sum_{k=1}^\infty\frac{1+ \chi(p^k)}{p^{k\sigma}} = \sum_{p\le N} \left( \frac{\log p}{p^{\sigma}-1}  + \frac{ \chi(p)  \log p}{p^{\sigma}- \chi(p) } \right)  
\end{equation}
after computing the two geometric series over $k$. Observe that the remaining sum is easily computed once the values of $\chi(p)$ for all primes $p \le N$ are known.  

Now, if $L(s,\chi)$ has a real zero $\beta_1 \in (0,1)$, then we apply  \Cref{thm:explicitvalues} and combine \eqref{eqn:non-negativity} and \eqref{eqn:primesum} to deduce that 
\begin{equation}
    \sum_{p\le N} \left( \frac{\log p}{p^{\sigma}-1}  + \frac{ \chi(p)  \log p}{p^{\sigma}- \chi(p) } \right)  
    \le 
    \frac{1}{r}-\frac{1}{\sigma-\beta_1}+\frac{\sigma-\beta_1}{R^2}+\phi\log q+E.
    \label{eqn:mainineq}
\end{equation}
If additionally $\beta_1\ge1-\frac{c}{\log q}$ for some absolute constant $c$ (namely $c=1/5$ in the case of \Cref{thm:main}), then \eqref{eqn:mainineq} becomes
\begin{equation} 
\label{eqn:quadIneq}
\sum_{p\le N} \left( \frac{\log p}{p^{\sigma}-1}  + \frac{ \chi(p)  \log p}{p^{\sigma}- \chi(p) } \right)  
\le \frac{c}{r(r\log q+c)}+\frac{r+c/\log q}{R^2}+\phi\log q+E.
\end{equation}
Our algorithm can now be explained in a simple manner: if we verify that \eqref{eqn:quadIneq} is false by computing both sides, then this proves $L(\sigma,\chi) \neq 0$ for $\sigma \ge 1-c/\log q$ as required by \eqref{eqn:runtime-problem}. 

Indeed, for a fixed $r$ (and hence fixed $\sigma, \phi, R,$ and $E$) from \Cref{table:ineq}, we will evaluate the left-hand sum over primes for a large enough value of $N$ until we establish that \eqref{eqn:quadIneq} is false. Assuming GRH, we establish in \S\ref{sect:runtime} that for sufficiently large $q$ depending on $c$, a large enough $N$ is guaranteed to exist provided that $r$ is well-chosen. 

\subsection{Runtime analysis}
\label{sect:runtime}
Assuming GRH, we provide a runtime analysis as $q \to \infty$ for our algorithm to solve \eqref{eqn:runtime-problem} for a fixed absolute constant $c > 0$. This requires two steps.
\begin{enumerate}
    \item First, we  determine the complexity of computing both sides of \eqref{eqn:quadIneq} for arbitrarily large moduli $q$, arbitrarily large $N$, and fixed values of $c, \sigma, r, \phi, R,$ and $E$.
    \item Second, we determine the required size of $N$ to violate \eqref{eqn:quadIneq} and hence establish \eqref{eqn:runtime-problem}. We will minimize the required size of $N$ relative to the fixed parameters $c, \sigma, r, \phi, R,$ and $E$, so $N$ will depend on $q$ as well as these fixed choices. 
\end{enumerate}  
The assumption of GRH will appear in step two. 

We begin with the first step. For the sake of simplicity, we shall assume that rigorous computation of elementary functions is $O(1)$, and  that getting the $n^{\mathrm{th}}$ prime is $O(1)$ by precomputing. Thus, the right-hand side of \eqref{eqn:quadIneq} requires $O(1)$ operations to compute. The left-hand side of \eqref{eqn:quadIneq} is the barrier since the sum runs over primes $p \le N$, and each summand requires the computation of $\chi(p)$.   Any primitive quadratic character is uniquely given by $\chi_d(n):=\big(\frac{d}{n}\big)_K$ where $\big(\frac{d}{n}\big)_K$ is the Kronecker symbol associated to fundamental discriminant $d$ and $|d|$ is the modulus of the character (see, e.g., \cite[Theorem 9.13]{Montgomery_Vaughan_2007}). Since we only require $\chi_d(p)$ at  primes $p$, this reduces to the Legendre symbol (for $p>2$), which can be computed in $O((\log d)(\log p))$ operations (see, e.g. \cite[Algorithm 1.4.10]{Cohen_1993}). As $q = |d|$, it follows that the left-hand side of \eqref{eqn:quadIneq} can be computed in $O\big( N (\log N) (\log q) \big)$ runtime and $O(N)$ memory for storing the primes $p \le N$. 

(In practice, we are a bit more efficient. The Legendre symbol is also $p$-periodic, allowing us to buffer $\big(\frac{\cdot}{p}\big)$. This saves the $O((\log d)(\log p))$ factor on some of the terms of the sum in \eqref{eqn:quadIneq} at a cost of $O(P^2)$ memory and $O_\epsilon (P^{2+\epsilon})$ precomputation time, where $P$ is the largest prime for which the Legendre symbol has been buffered. In the actual computation performed, we buffer these values up to the first 10,000 primes due to memory constraints, but this does not affect the asymptotic runtime of the program.)

  We proceed to the second step to determine the minimum size of $N$ required to violate \eqref{eqn:quadIneq}. Our argument rests on a routine  conditional estimate for the sum over primes. 

\begin{lemma}
\label[lemma]{lem:runtimeBound}
    Let $\chi \pmod{q}$ be a primitive quadratic character. Assume GRH for $L(s,\chi)$.  Fix  constants $\theta, \lambda > 0$. Let $N=q^\theta$ and $\sigma = 1+\frac{\lambda}{\log q}$. Then
    \begin{equation}
    \label{eq:runtimeBound}
    \sum_{p\le N} \left( \frac{\log p}{p^{\sigma}-1}  + \frac{ \chi(p)  \log p}{p^{\sigma}- \chi(p) } \right)   = \frac{1-e^{-\lambda \theta}}{\lambda}\log q + O_{\theta,\lambda}(\log\log q).    
    \end{equation}
\end{lemma}
\begin{proof}
 By the Prime Number Theorem and partial summation, we have that 
    $$\sum_{n \le N}\Lambda(n)\frac{1}{n^\sigma}=\frac{1-N^{-\lambda/\log q}}{\lambda}\log q+O(1).$$
    Let $M=(\log q)^3$. Using $|\chi|\le1$ and the Prime Number Theorem again, it follows that 
    \begin{align*}
        \sum_{n \le M}\Lambda(n)\frac{\chi(n)}{n^\sigma} &\ll \sum_{n \le M}\Lambda(n)\frac{1}{n} 
        \ll \log M \ll \log \log q.
    \end{align*}
    The function $\psi(x,\chi) := \sum_{n \le x} \chi(n) \Lambda(n)$ for $x \ge 1$ satisfies $\psi(x,\chi) \ll \sqrt{x}(\log x)(\log qx)$ by GRH (e.g. \cite[Theorem 13.7]{Montgomery_Vaughan_2007}). Therefore, by partial summation, 
    \begin{align*}
        \sum_{M<n\le N}\Lambda(n)\frac{\chi(n)}{n^\sigma}&=\frac{\psi(x,\chi)}{x^\sigma}\Big|^N_M+\sigma\int_M^N \psi(x,\chi)x^{-\sigma-1}dx\\
        & \ll  M^{1/2-\sigma} \log M  \log (Mq) 
     +   \int_M^N x^{-\sigma-1/2}(\log x)(\log qx)dx  
 \ll \frac{\log\log q}{(\log q)^{1/2}}.  
    \end{align*}
    Since 
    \[
    \sum_{p\le N} \left( \frac{\log p}{p^{\sigma}-1}  + \frac{ \chi(p)  \log p}{p^{\sigma}- \chi(p) } \right)   =  \sum_{p\le N}(\log p)\sum_{k=1}^\infty\frac{1+\chi(p^k)}{p^{k\sigma}} =   \sum_{n\le N}\Lambda(n)\frac{1+\chi(n)}{n^\sigma} + O(1),
    \]
    the result follows. 
\end{proof}

Write $N = q^{\theta}$ for some parameter $\theta > 0$ yet to be optimized. We have shown that the evaluation of both sides of \eqref{eqn:quadIneq} has runtime $\ll\pi(q^\theta) \log(q)^2\ll_{\epsilon} q^{\theta+\epsilon}$. By \Cref{lem:runtimeBound} (and recalling $\sigma=1+r=1 + \frac{\lambda}{\log q}$), we will establish \eqref{eqn:runtime-problem} provided 
\begin{equation}
\label{eq:asymptoticInequality}
\frac{1-e^{-\theta\lambda}}{\lambda} \log q > \frac{c\log q}{\lambda(\lambda+c)}+ \phi \log q + O(\log\log q)    .
\end{equation}
Dividing both sides by $\log q$ and taking $q \to \infty$, this inequality will hold for large enough $q$ provided 
\[
\frac{1-e^{-\theta\lambda}}{\lambda}  > \frac{c}{\lambda (\lambda+c)} + \frac{1}{4}
\]
since $\phi\le 1/4$ (see \Cref{rem:choice-of-constants}).  Rearranging, the above is equivalent to the statement
\begin{equation*}
\label{eq:thetaInequality}
\theta> -\frac{\log(1-c/(\lambda+c)-\lambda/4)}{\lambda} =:\theta_c(\lambda).
\end{equation*}
It follows that for a given value of $c$ in \eqref{eqn:runtime-problem}, the GRH-conditional runtime for our program is $O_{\epsilon}(q^{\theta_c(\lambda) + \epsilon})$ for any $\lambda$. Thus, for any given value of $c$, it remains to minimize $\theta_c(\lambda)$ over $\lambda$. 

For $c= 1/5$, $\theta_c(\lambda)$ is minimized at $\lambda = \lambda_{\mathrm{min}}$ when $\lambda_{\mathrm{min}}\approx1.8$ with $ \theta_c(\lambda_{\mathrm{min}})<0.444$. Therefore, the asymptotic runtime for our program is $O(q^{0.444})$ for $q$ sufficiently large. This can be improved to $O(q^{0.384})$ when $c=1/10$ instead. We also observe that the minimal value of $\theta_c(\lambda)$ approaches $1/4$ when $c \to 0^+$, resulting in a runtime $O_\epsilon(q^{1/4+\epsilon})$ for $c \le c(\epsilon)$ sufficiently small and $q \geq q(c)$ sufficiently large. 

This concludes the explanation of our claimed theoretical runtimes in the introduction. In practice, our range of moduli $q \le 10^{10}$ is still relatively small on the $\log q$ scale in \Cref{lem:runtimeBound}, so the error term $O(\log \log q)$ therein and its implied constants can play a noticeable role for some moduli $q$. These variations for certain moduli are described in \S\ref{sect:proof-main}.

\section{Proofs of Theorem 1.1 and Corollary 1.3}
\label{sect:proof-main}

\begin{proof}[Proof of \Cref{thm:main} assuming \Cref{thm:explicitvalues}]
For $q\le 4\times 10^5$, we apply Platt's result \cite{Platt_2016}.

For $4 \times 10^5 < q \leq 10^{10}$, we begin with two initial simplifications for our computation. First, rather than have $r$ vary with $q$, we will set 
\begin{equation*} 
r=\frac{\lambda}{\log (10^{10})}
\end{equation*}
for some fixed $\lambda > 0$ in order to reduce computations, since $10^{10}$ is the largest modulus we will consider. Our inequality might then be suboptimal for smaller moduli, but we are primarily concerned with increasing efficiency for large moduli. To sum over primes, we buffered a list of the first 130,000,000 primes, which was generated using the prime sieve \cite{primesieve}. 

Second, as explained in \S\ref{sect:runtime}, we will loop through fundamental discriminants and calculate the Kronecker symbol. Recall that $d$ is a fundamental discriminant iff $d$ satisfies 
\begin{itemize}
    \item[(a)] $d\equiv 1\textnormal{ (mod } 4)$ and $d$ is square-free or
    \item[(b)] $d\equiv 0\textnormal{ (mod } 4)$, $d/4\equiv 2$ or $3\textnormal{ (mod } 4)$, and $d/4$ is square-free.
\end{itemize}
However, we chose not to implement any square-free checks, so we instead loop through all $d$ satisfying 
\begin{equation}
    d\equiv 1\textnormal{ (mod } 4) \textnormal{ or } d\equiv 8 \textnormal{ or } 12\textnormal{ (mod 16}).
    \label{eqn:dCondition}
\end{equation}    
In this way, we over-check and verify that \eqref{eqn:quadIneq} does not hold for some $\chi$ which are not primitive quadratic characters. (One can choose to implement a check to rule out small square factors or even implement the full square-free check, which could possibly improve performance.) We also precomputed $\big(\frac{\cdot}{p}\big)$ for the first 10,000 primes $p$.

Now, we proceed to the computation. Using the explicit inequalities from \Cref{thm:explicitvalues}, we used the FLINT library \cite{flint} in C to compute the left side of \eqref{eqn:quadIneq} for all $4\times 10^5< q\le 10^{10}$ and verified that \eqref{eqn:quadIneq} does not hold. To make the computation rigorous, we also used arbitrary-precision ball arithmetic \cite{ball} to obtain rigorous error bounds. (Ball arithmetic is actually strictly stronger than what is necessary here; one could also simply round down the left side and round up the right side in \eqref{eqn:quadIneq}.)

To computationally verify that \eqref{eqn:quadIneq} does not hold, the program takes multiple runs. For each run, we fix some cutoff value $N_0$ and some value for $r$, then the program loops through values of $d$ satisfying \eqref{eqn:dCondition} and computes the sum on the left side of \eqref{eqn:quadIneq} with increasing $N$ until either the inequality is violated or we have summed over $N_0$ primes (in this case, we say that $d$ has `failed'). The program takes multiple passes through the failed $d$ with higher and higher values of $N_0$ and re-optimized values of $r$ that minimize the tail above the cutoff $N_0$. The process of optimizing $r$ for various values of $N_0$ will be explained in \S\ref{sect:optimizeR}. Intuitively, smaller values of $r$ make the sum on the left side of \eqref{eqn:mainineq} larger, allowing us to violate the inequality for more values of $d$; but they also `spread out the weight' of the sum towards the tail, so it takes larger values of $N$ to notice the gains.

We used $\approx150$,000 core hours on  SHARCNET's Nibi cluster as follows. First, we ran the program on all $4 \times 10^5< |d|\le 10^{10}$ satisfying \eqref{eqn:dCondition} using Row 1 of \Cref{table:ineq} up to a cutoff of $N_0=6$0,000. This step took 42,300 core hours (or 4.6 days) to run with 2 nodes, 192 CPU cores per node, and 754 GB of memory per node. After this run, there were 1,536,586,624 failed moduli. For the second run, we looped through the failed $d$ using Row 2 of \Cref{table:ineq} up to a cutoff of $N_0=5$,000,000. This step took 105,000 core hours (or 5.7 days) to run on 4 nodes, 192 CPU cores per node, and resulted in 1,148,834 failed moduli. For the third run, we looped through the $d$ that failed both previous runs using Row 3 of \Cref{table:ineq} with a cutoff of $N_0=5$0,000,000. This step took 2,000 core hours (or 11.2 hours) on 1 node with 192 CPU cores, resulting in only $100$ failed moduli. For the last run, we used Row 4 of \Cref{table:ineq} with a cutoff of $N_0=1$30,000,000. This step took 3 core hours (or 9.4 minutes) to run on 20 CPU cores with 6 GB of memory per core, and it successfully verified all remaining moduli.  
\end{proof}

\begin{figure}[ht]
    \centering
    \includegraphics[width=1\linewidth]{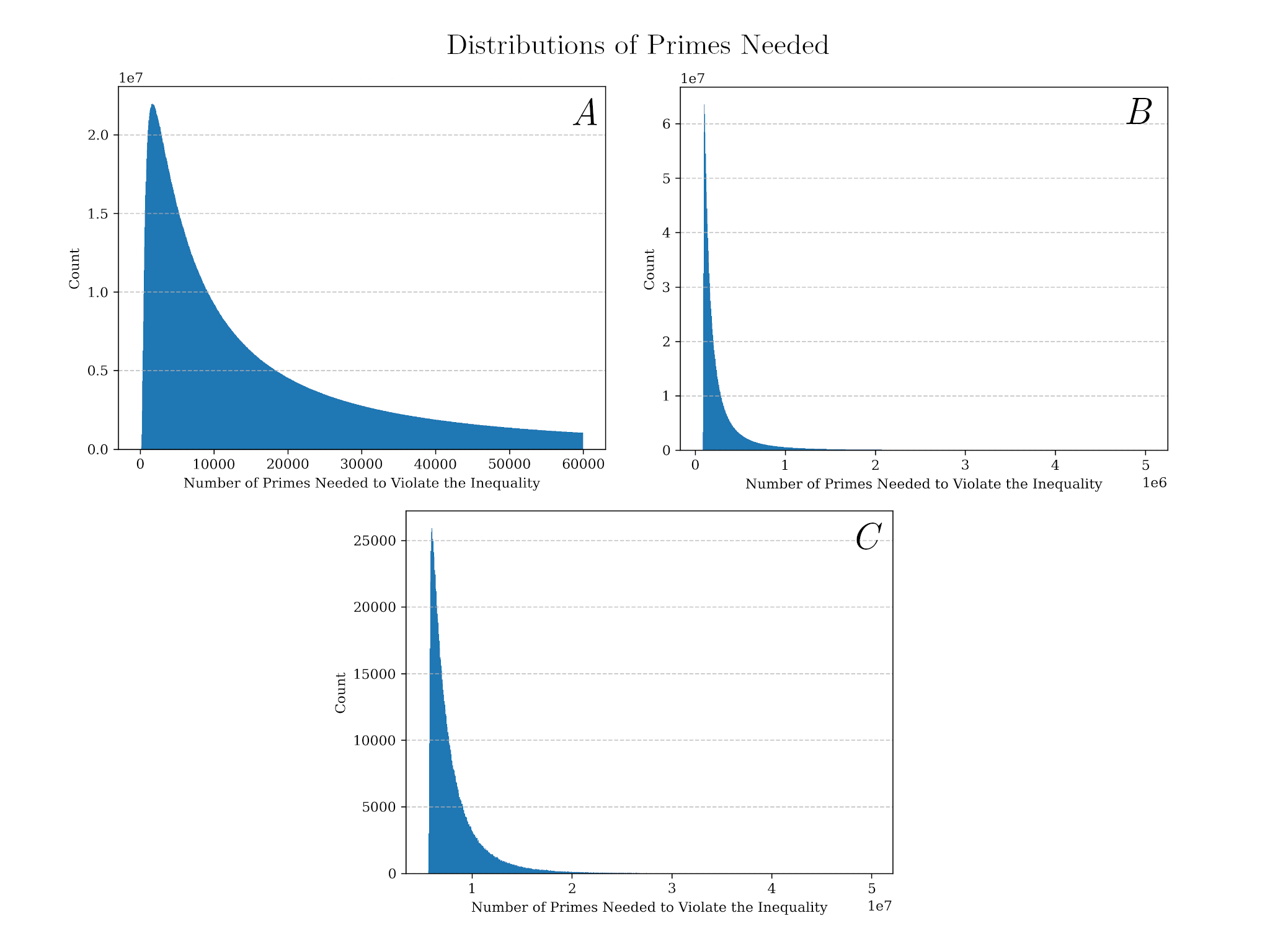}
    \caption{Histogram of number of primes needed to violate \eqref{eqn:quadIneq} in the first (A), second (B), and third (C) runs, with bucket sizes $50$, 5,000, and 50,000 respectively. (A) has median 10,150, mean 15,519, and standard deviation 14,599. (B) has median 175,000, mean 285,803, and standard deviation 339,024. (C) has median 7,100,000,  mean 8,040,486, and standard deviation 2,845,708.
}
    \label{fig:run}
\end{figure}

The distribution of the number of primes needed to run these computations can be found in \Cref{fig:run}. The median number of primes it took to violate the inequality was 15,350. However, in the worst case, it may take millions of primes. So for a given $d$, our method is quite efficient most of the time but very slow for some exceptional $d$. We expect the relative efficiency of our method to increase with the modulus because the $O(\log\log q)$ term in \eqref{eq:asymptoticInequality} becomes insignificant for large $q$. The code also runs much faster with smaller values of $c$. (For $c=1/10$, it took 19,200 core hours (or 50 hours) in total on 2 nodes to perform the same computation for all $4\times 10^{5}< q\le 10^{10}$. See also \S\ref{sect:runtime}.)

\begin{proof}[Proof of \Cref{cor:L1chi}]
 For $q \le 10^6$, the inequality was computationally verified using the default Dirichlet $L$-function algorithm in FLINT (see our GitHub code at 
    \href{https://github.com/asif-z/Landau-Siegel-Zero-Tester/blob/master/L1/L1.c}{\tt L1/L1.c}). Hoffstein \cite[Lemma 1]{Hoffstein_1980} showed that if $q > 10^6$ and $L(\sigma,\chi) \neq 0$ for $\sigma > 1 -1/(11.657 \log q)$ then $L(1,\chi) > 1.507 / (11.657 \log q)$. The result follows by \Cref{thm:main}  as $1.507/11.657 > 1/8$. 
\end{proof}

\section{Preliminaries with logarithmic derivatives}
\label{sect:preliminary}
The rest of the paper is dedicated to the proof of \Cref{thm:explicitvalues}, an explicit inequality involving the logarithmic derivatives of the Riemann zeta function $\zeta(s)$ and the Dirichlet $L$-function $L(s,\chi)$ of a primitive quadratic character $\chi$. This section begins with some preliminary steps.  The Riemann zeta term $-\frac{\zeta'}{\zeta}(\sigma)$ is bounded via  its explicit formula.
\begin{lemma}
\label[lemma]{thm:zetabound}
If $r>0$, then
$$-\ld\zeta(1+r)\le \frac{1}{r}-\log 2\pi+\frac{\gamma}{2}+1+\frac{1}{2}\ld{\Gamma}\left(\frac{3+r}{2}\right),$$
where $\gamma=0.577\dots$ is the Euler-Mascheroni constant and $\Gamma$   is the Gamma function.
\end{lemma}
\begin{proof}
The explicit formula (e.g. \cite[Corollary 10.14]{Montgomery_Vaughan_2007}) shows that 
$$-\ld\zeta(1+r)=-B-\frac{1}{2}\log\pi+\frac{1}{r}+\frac{1}{2}\ld\Gamma\left(1+\frac{1+r}{2}\right)-\Re\sum_{\rho}\left(\frac{1}{1+r-\rho}+\frac{1}{\rho}\right)$$
where $\rho$ ranges over non-trivial zeros of $\zeta$ and $B=-\frac{\gamma}{2}-1+\frac{1}{2}\log 4\pi$. Note we can drop the sum because both $1+r-\rho$ and $\rho$ have positive real part, hence 
$$\Re\left(\frac{1}{1+r-\rho}+\frac{1}{\rho}\right)\ge 0.$$ 
The inequality follows. 
\end{proof}

The logarithmic derivative $-\frac{L'}{L}(\sigma,\chi)$ can also be bounded with its explicit formula, but it will be too inefficient. Instead, we will use a version of Jensen's formula and convexity estimate; similar strategies appear in \cite[Lemma 3.1]{Heath-Brown_1992} and \cite[Proposition 2.11]{Thorner_Zaman_2024}.

\begin{lemma}[Jensen's Formula] 
\label[lemma]{lem:jensen}
If $f(z)$ is analytic for $|z-a|\le r_1$ and non-vanishing at $z=a$ and on the circle $|z-a|=r_1$, then
$$\Re\ld{f}(a)=\Re\sum_{\rho}\left(\frac{1}{a-\rho}-\frac{a-\rho}{r_1^2}\right)+\frac{1}{\pi r_1}\int_0^{2\pi}(\cos\theta)\log\left|f(a+r_1e^{i\theta})\right|d\theta$$
where $\rho$ ranges over zeros of $f$ in the disc $|z-a|< r_1$, counting with multiplicity.
\end{lemma}
\begin{proof}
See \cite[Lemma 3.2]{Heath-Brown_1992}.
\end{proof}

\begin{lemma}
\label[lemma]{thm:boundJ} Let $\chi \pmod{q}$ be a primitive character. 
Suppose $L(s,\chi)$ has a real zero $\beta_1\in (0,1)$. For all $r>0$ and $r_1>0$ such that $L(s,\chi)$ has no zeros on $|1+r-s|=r_1$, we have
    $$-\Re\ld{L}(1+r,\chi) \le \frac{1+r-\beta_1}{r_1^2}-\frac{1}{1+r-\beta_1}- J$$
    where
\begin{equation} \label{eqn:J}
    J=\int_0^{2\pi}\frac{1}{\pi r_1} (\cos\theta)\log|L(1+r+r_1e^{i\theta},\chi)|d\theta.
\end{equation}
\end{lemma}
\begin{proof}
    Apply \Cref{lem:jensen} to $f(z) = L(z,\chi)$ and $a=1+r$. For all zeros $\rho=\beta+i\gamma$ in the disc $|1+r-\rho| < r_1$, we have that  
    $$\Re\left(\frac{1}{1+r-\rho}-\frac{1+r-\rho}{r_1^2}\right)=(1+r-\beta)\left(\frac{1}{|1+r-\rho|^2}-\frac{1}{r_1^2}\right)\ge 0$$
    since $\beta\in[0,1)$.
    Hence, we may drop all of them except $\beta_1$ by non-negativity to see that 
    $$-\Re\sum_{\rho}\left(\frac{1}{1+r-\rho}-\frac{1+r-\rho}{r_1^2}\right)\le-\frac{1}{1+r-\beta_1}+\frac{1+r-\beta_1}{r_1^2}.$$
    The result follows. 
\end{proof}

\Cref{thm:boundJ} reduces our entire analysis to estimating the integral $J$. The right half of the integral with $\theta\in[-\pi/2,\pi/2]$  lies in the region of absolute convergence for $L(s,\chi)$, so it is easily treated with the following elementary bound.  
\begin{lemma}
\label[lemma]{thm:J0bound}
    Let $r>0$. Suppose $R=ar+b$ where $a > 0$ and $0 < b\le1$. If
    $$J_0=\frac{1}{\pi R}\int_{-\pi/2}^{\pi/2}(\cos\theta)\log |L(1+r+Re^{i\theta},\chi)|d\theta$$
    then 
    $$|J_0|\le \frac{1}{\pi R}\left(\frac{\pi}{b}-4\sqrt{b^{-2}-1}\tan^{-1}\left(\sqrt\frac{1-b}{1+b}\right)-2\log(2b)\right).$$
\end{lemma}
\begin{proof}
    For $\Re s>1$,
    $$\left|\log\left|L(s,\chi)\right|\right|=|\Re(\log L(s,\chi))|\le |\log L(s,\chi)|\le\log\zeta(\Re s)\le\log\left(1+\frac{1}{\Re s-1}\right).$$
    This implies that 
    \begin{align*}
        |J_0|&\le\frac{1}{\pi R}\int_{-\pi/2}^{\pi/2}(\cos\theta)\log\left(1+\frac{1}{r+R\cos\theta}\right)d\theta\\
        &\le\frac{1}{\pi R}\int_{-\pi/2}^{\pi/2}(\cos\theta)\log\left(1+\frac{1}{b\cos\theta}\right)d\theta\\
        &=\frac{2}{\pi R}\lim_{\epsilon\rightarrow0^+}\Bigg[\Bigg(\frac{\theta}{b}-2\sqrt{b^{-2}-1}\tan^{-1}\Bigg(\sqrt{\frac{1-b}{1+b}}\tan(\theta/2)\Bigg)-\log(1+\tan(\theta/2))\\
        &\quad+\log(1-\tan(\theta/2))+\sin(\theta)\log\Big(1+\frac{1}{b\cos\theta}\Big)\Bigg)\Bigg]^{\theta=\pi/2-\epsilon}_{\theta=0}\\
        &=\frac{1}{\pi R}\left(\frac{\pi}{b}-4\sqrt{b^{-2}-1}\tan^{-1}\left(\sqrt{\frac{1-b}{1+b}}\right)-2\log 2+2\alpha\right)
    \end{align*}
    where 
    \begin{align*}
        \alpha&:=\lim_{\theta\rightarrow\pi/2^-}\left(\log(1-\tan(\theta/2))+\sin(\theta)\log\left(1+\frac{1}{b\cos\theta}\right)\right)\\
        &=\lim_{\theta\rightarrow\pi/2^-}\left(\log\left(\frac{1-\tan(\theta/2)}{\cos\theta}\cdot\frac{b\cos\theta+1}{b}\right)\right)+\lim_{\theta\rightarrow\pi/2^-}\left(\left(\sin\theta-1\right)\log\left(1+\frac{1}{b\cos\theta}\right)\right)\\
        &=\log\left(b^{-1}\right)+0.
    \end{align*}
   The result follows.
\end{proof}

\begin{remark}
    Since only finitely many $a,b$, and $r$ are considered in the actual computation, some of the above inequalities can be replaced by rigorous numerical integrations involving the zeta function, which could improve the numerical constants we obtained in \Cref{table:ineq}.
\end{remark}

The left half of the integral $J$ in \eqref{eqn:J} lies inside the critical strip for $L(s,\chi)$, so an explicit bound for $\log|L|$ will be more involved. This problem will be our focus in the next section.  
\section{Convexity bounds via an iterated Phragmén–Lindelöf principle}
\label{sect:convexity}

To bound $\log|L|$ to the left of the $\Re(s)=1$ line, we will use the following explicit convexity bound for $L(s,\chi)$ along vertical lines in addition to the Phragmén–Lindelöf principle. 

\begin{lemma}[Thorner and Zaman]
\label[lemma]{thm:12bound} Let $q \ge 3$. 
   If $\chi \pmod{q}$ is primitive then, for $ t \in \R$, 
    $$\left|L\left(\frac{1}{2}+it,\chi\right)\right|\le C_2(q|1+it|)^{1/4} \quad \text{ with }  C_2=2.97655.$$
\end{lemma}
\begin{proof}
    See \cite[Proposition 2.10]{Thorner_Zaman_2024}.
\end{proof}

The convexity exponent of $1/4$ here is of utmost importance. It contributes the factor of $\phi$ in \Cref{thm:explicitvalues}, and hence the exponent on the runtime of our algorithm as $c\rightarrow 0^+$ in \eqref{eqn:runtime-problem}. In principle, we can strengthen $\phi$ to correspond to the Weyl subconvexity exponent of $1/6+\epsilon$  by following Heath-Brown's arguments \cite[Lemma 3.1]{Heath-Brown_1992} (see also the landmark results of Burgess \cite{Burgess}, Conrey--Iwaniec \cite{ConreyIwaniec-CubicMomentAutomorphic}, and Petrow--Young \cite{PetrowYoung-Weyl,PetrowYoung-WeylCoset}). However, as far as we are aware, existing explicit subconvexity estimates appear insufficient for our purposes due to restrictions on the moduli or large implied constants, e.g. \cite{JainSharmaKhaleLiu-ExplicitBurgess,BordignonFrancis-ExplicitBurgess}. This challenge appears somewhat inherent to those methods, so we rely on convexity bounds.  

While we could combine \Cref{thm:12bound} and the classical Phragmén–Lindelöf principle \cite{Rademacher_1959} to estimate $L(s,\chi)$ uniformly for $1/2 \le \Re(s) \le 1$, we instead establish refined bounds along $\Re (s)=\frac{3}{4}, \frac{7}{8},$ and $\frac{15}{16}$. The goal of this section is to establish the following key proposition. 
\begin{proposition}
\label[proposition]{thm:34bound}
    Let $q \ge 3$. If $\chi \pmod{q}$ is primitive then, for $t \in \R$,  
    $$\left|L\left(\frac{3}{4}+it,\chi\right)\right|\le C_4(q|1+it|)^\frac{1}{8} \quad \text{with } C_4=5.13781, $$
    $$\left|L\left(\frac{7}{8}+it,\chi\right)\right|\le C_8(q|1+it|)^\frac{1}{16}  \quad \text{with } C_8=9.49562,$$
    and
    $$\left|L\left(\frac{15}{16}+it,\chi\right)\right|\le C_{16}(q|1+it|)^\frac{1}{32}  \quad \text{with } C_{16}=18.23874.$$
\end{proposition}
The proof iterates the technique of Thorner--Zaman \cite[Proposition 2.10]{Thorner_Zaman_2024} originating from Heath-Brown \cite{HeathBrown2009} via the following lemma whose proof we temporarily postpone. 
\begin{lemma} Let $q \ge 3$ and $\chi \pmod{q}$ be a primitive character. Let $k \ge 1$ be an integer and $C > 0$ be a positive real number. Assume 
\begin{equation}
    \left|L\left(1-\frac{1}{2^k}+it,\chi\right)\right|\le C(q|1+it|)^{1/2^{k+1}}
    \label{eqn:thmassumption}
\end{equation}
for all $t \in \R$.  Then
$$\left|L\left(1-\frac{1}{2^{k+1}}+it,\chi\right)\right|\le \sqrt{C}e^{P/2^{k+3}} \left( \prod_{j=0}^\infty \zeta\left(1+\frac{1}{2^{k+1}}+\frac{j}{2^k}\right)^{(-1)^j} \right) (q|1+it|)^{1/2^{k+2}}$$
for all $t \in \R$, where
$$P=\sup_{t\in\mathbb{R}}\frac{1}{(t^2+1)^3}\left(-\frac{1}{4\cdot2^{2k}}t^4+\frac{35}{32\cdot 2^{4k}}t^2+\frac{1}{4\cdot 2^{2k}}-\frac{5}{32\cdot 2^{4k}}+\frac{61}{192\cdot 2^{6k}}\right).$$
\label[lemma]{thm:Lbound}
\end{lemma}

\begin{remark}
    The supremum in the definition of $P$ above is attained at $t=0$. Indeed, by differentiating $$h(t):=\frac{1}{(t^2+1)^3}\left(-\frac{1}{4\cdot2^{2k}}t^4+\frac{35}{32\cdot 2^{4k}}t^2+\frac{1}{4\cdot 2^{2k}}-\frac{5}{32\cdot 2^{4k}}+\frac{61}{192\cdot 2^{6k}}\right),$$ we see that $h'(t)$ has three simple zeros at $0,\pm t_0$ for some $t_0>0$. Together with the fact that $h'(t)>0$ for large $t$, $h(t)\to 0$ as $|t|\to \infty$, and $h(0)>0$, the result follows immediately.  
\end{remark}

Assuming the above lemma, we may immediately deduce \Cref{thm:34bound}. 

\begin{proof}[Proof of \Cref{thm:34bound}]
We evaluate the values in \Cref{thm:Lbound} explicitly, using \Cref{thm:12bound} for the first iteration at $k=1$. Note that the log of the infinite product is an alternating series with terms decreasing in absolute value. Hence, we can truncate it for an upper bound at an odd number of terms. Thus, by evaluating the sum for $0\le j \le 1000$, we obtain the theorem. The code to perform these calculations can be found in our GitHub at \href{https://github.com/asif-z/Landau-Siegel-Zero-Tester/blob/master/misc/rigorousBound.c}{\tt misc/rigorousBound.c}.
\end{proof}

 It remains to prove \Cref{thm:Lbound}. The key input is a lemma of Bennett and Sharpley \cite{Bennett_Sharpley_1988} which interpolates operator bounds. Given estimates for a complex analytic function of order $1$ along two vertical lines, this method can  provide sharper bounds than classical versions of the Phragmén–Lindelöf principle \cite{Rademacher_1959}, but it only provides an estimate along the  vertical line halfway between the original two lines rather than the entire strip between them. 
\begin{proof}[Proof of \Cref{thm:Lbound}]
    Apply  \cite[Chapter 3, Lemma 3.1]{Bennett_Sharpley_1988} with $F(z)=L(1-\frac{1}{2^k}+\frac{z}{2^k}+it,\chi)$ and $x=1/2$, then we obtain
    \begin{align*}
        \log\left|L\left(1-\frac{1}{2^{k+1}}+it,\chi\right)\right|
        &\le\frac{1}{2}\int_{-\infty}^\infty\frac{\log|L(1-\frac{1}{2^k}+i(t+\frac{y}{2^k}),\chi)L(1+i(t+\frac{y}{2^k}),\chi)|}{\cosh(\pi y)}dy\\
        &= \frac{1}{2}I_1+\frac{1}{2}I_2
    \end{align*}
    where
    $$I_1=\int_{-\infty}^\infty\log\left|L\left(1-\frac{1}{2^k}+i(t+y/2^k),\chi\right)\right|\frac{dy}{\cosh(\pi y)}$$
    and
    $$I_2:=\int_{-\infty}^\infty\log\big|L(1+i(t+y/2^k),\chi)\big|\frac{dy}{\cosh(\pi y)}.$$
    Now, using assumption \eqref{eqn:thmassumption} for $L$, we have 
    \begin{equation}
    I_1\le \log C+\frac{1}{2^{k+1}}\log q+\frac{1}{2^{k+1}}\int_{-\infty}^\infty\log\big|1+i(t+y/2^k)\big|\frac{dy}{\cosh(\pi y)}
    \label{eqn:firstIntegral}.
    \end{equation}
    To treat the integral, note that
    $$\big|1+i(t+y/2^k)\big|=|1+it|\sqrt{\left(1+\frac{(y/2^k)^2+2ty/2^k}{t^2+1}\right)}$$
    and
    \begin{align*}
        \int_{-\infty}^\infty \log\left(1+\frac{(y/2^k)^2+2ty/2^k}{t^2+1}\right)\frac{dy}{\cosh(\pi y)}&\le\int_{-\infty}^\infty \sum_{j=1}^3\frac{(-1)^{j+1}}{j}\left(\frac{(y/2^k)^2+2ty/2^k}{t^2+1}\right)^j\frac{dy}{\cosh(\pi y)}\\
        &\le P
    \end{align*}
    where we have used an alternating series truncation of $\log$ in the form $\log(1+\omega)\le\omega-\frac{\omega^2}{2}+\frac{\omega^3}{3}$ for $\omega>-1$ and
    $$v^2+2tv=(v+t)^2-t^2\ge-t^2>-(t^2+1)$$
    for any $v,t\in \R$. Hence, the integral in \eqref{eqn:firstIntegral} is
    $\le \log|1+it|+\frac{1}{2}P.$

    As for $I_2$, applying dominated convergence, we get
    \begin{align*}
        I_2&=\lim_{\xi\rightarrow 0^+}\int_{-\infty}^\infty\log\big|L(1+\xi+i(t+y/2^k),\chi)\big|\frac{dy}{\cosh(\pi y)}\\
        &=\lim_{\xi\rightarrow0^+}\Re\sum_{n=2}^\infty\frac{\Lambda(n)\chi(n)}{n^{1+\xi+it}\log n}\int_{-\infty}^\infty n^{-iy/2^k}\frac{dy}{\cosh(\pi y)}.
    \end{align*}
    By contour integration, the inner integral is equal to $\frac{2n^{1/2^{k+1}}}{1+n^{1/2^k}}$. Moving the limit inside the sum by dominated convergence again, we obtain
    \begin{align*}
        I_2&=\lim_{\xi\rightarrow0^+}2\Re\sum_{n=2}^\infty\frac{\Lambda(n)\chi(n)}{(n^{1+1/2^{k+1}+\xi+it}+n^{1-1/2^{k+1}+\xi+it})\log n}\\
        &\le 2\sum_{n=2}^\infty\frac{\Lambda(n)}{(\log n)n^{1+1/2^{k+1}}(1+n^{-1/2^k})}\\
        &=2\sum_{n=2}^\infty\frac{\Lambda(n)}{\log n}\sum_{j=0}^\infty\frac{(-1)^j}{n^{1+1/2^{k+1}+j/2^k}}\\
        &=2\sum_{j=0}^\infty(-1)^j\log\zeta\left(1+\frac{1}{2^{k+1}}+\frac{j}{2^k}\right).
    \end{align*}
    The inequality follows. 
\end{proof}

We close this section with a final remark about our future application of \Cref{thm:34bound}. For our purposes, we will not use our estimate along $\Re s=\frac{15}{16}$ or any further iterations of these bounds. Indeed, one can verify that the bound
\begin{equation}
    \left|L\left(\frac{15}{16}+it\right)\right|\le \Big(C_8\big(q|1+it|\big)^{1/16}\Big)^{\frac{r+1/16}{r+1/8}}(1+r^{-1})^{\frac{1}{2+16r}}
    \label{eqn:1516bound}
\end{equation}
follows from applying the classical Phragmén–Lindelöf principle to our estimate along $\Re s=7/8$  from \Cref{thm:34bound}  and the trivial bound $|L(s,\chi)| \le \zeta(\Re s)\le 1+r^{-1}$ along $\Re s=1+r$. Although worse asymptotically, estimate \eqref{eqn:1516bound} is better than \Cref{thm:34bound} along $\Re s = \frac{15}{16}$ when $q=10^{10}$, $t=0$, and $r=\frac{\lambda}{10\log 10}$,  for $\lambda \in[0.7, 20]$. For example, at $\lambda=1.6,$ 1.3, 1.2, and 1.1, the bound in \eqref{eqn:1516bound} becomes
$$\left|L\left(\frac{15}{16}+it\right)\right|\le C(q|1+it|\big)^\frac{r+1/16}{16r+2}$$
for $C=11.1$, 12.0, 12.4, and 12.9, respectively. Thus, we do not expect much improvement to our result with the $\Re s=15/16$ bound, so we terminate the iteration at $\Re s=7/8$.

\section{Final estimates with logarithmic derivatives}
\label{sect:bound}
In this section, we assemble our estimates to establish the key inequalities with logarithmic derivatives which will be subsequently used to prove \Cref{thm:explicitvalues}. For the entirety of this section, let $\chi$ be a primitive character of modulus $q \ge 3$.

\subsection{Inside the critical strip}  Given  \Cref{thm:boundJ,thm:J0bound}, our focus is to estimate the integral $J$ from \eqref{eqn:J}, especially the portion corresponding to values of $L(s,\chi)$ lying inside the critical strip. Keeping the same notation as those lemmas, fix $r\in(0,1/4)$ and let $r_1\in [r+3/4,r+7/8]$.  We will eventually take $r_1\rightarrow r+7/8$. 

\tikzset{every picture/.style={line width=0.75pt}} 

\begin{figure}[h!]
\centering  

\begin{tikzpicture}[x=0.75pt,y=0.75pt,yscale=-.8,xscale=.8]

\draw  (1.22,241.54) -- (658.82,241.54)(120.42,2.34) -- (120.42,479.14) (651.82,236.54) -- (658.82,241.54) -- (651.82,246.54) (115.42,9.34) -- (120.42,2.34) -- (125.42,9.34)  ;
\draw   (140,240.61) .. controls (140,136.67) and (224.27,52.4) .. (328.21,52.4) .. controls (432.16,52.4) and (516.42,136.67) .. (516.42,240.61) .. controls (516.42,344.56) and (432.16,428.82) .. (328.21,428.82) .. controls (224.27,428.82) and (140,344.56) .. (140,240.61) -- cycle ;
\draw  [dash pattern={on 4.5pt off 4.5pt}]  (280.42,2) -- (280.42,480) ;
\draw  [dash pattern={on 0.84pt off 2.51pt}]  (199.62,2) -- (200.42,480) ;
\draw  [dash pattern={on 0.84pt off 2.51pt}]  (241.22,2) -- (242.02,480) ;
\draw  [dash pattern={on 0.84pt off 2.51pt}]  (160.42,2) -- (161.22,480) ;
\draw  [dash pattern={on 0.84pt off 2.51pt}]  (261.22,2) -- (262.02,480) ;
\draw  [dash pattern={on 0.84pt off 2.51pt}]  (138.82,2) -- (139.62,480) ;
\draw    (261.22,65.54) -- (328.21,240.61) ;
\draw    (199.74,103.9) -- (328.21,240.61) ;
\draw    (241.34,74.3) -- (328.21,240.61) ;
\draw    (160.14,155.9) -- (328.21,240.61) ;
\draw  [dash pattern={on 0.84pt off 2.51pt}]  (327.81,2) -- (328.61,480) ;
\draw  [color={rgb, 255:red, 0; green, 0; blue, 0 }  ,draw opacity=0.45 ][fill={rgb, 255:red, 248; green, 184; blue, 28 }  ,fill opacity=0.25 ] (288.86,240.44) .. controls (288.93,223.54) and (299.45,209.13) .. (314.23,203.46) -- (328.21,240.61) -- cycle ;
\draw  [color={rgb, 255:red, 0; green, 0; blue, 0 }  ,draw opacity=0.45 ][fill={rgb, 255:red, 248; green, 184; blue, 28 }  ,fill opacity=0.25 ] (298.89,240.48) .. controls (298.93,229.4) and (305.02,219.75) .. (313.99,214.71) -- (328.21,240.61) -- cycle ;
\draw  [color={rgb, 255:red, 0; green, 0; blue, 0 }  ,draw opacity=0.45 ][fill={rgb, 255:red, 248; green, 184; blue, 28 }  ,fill opacity=0.25 ] (305.22,240.51) .. controls (305.25,234.01) and (307.92,228.14) .. (312.21,223.95) -- (328.21,240.61) -- cycle ;
\draw  [color={rgb, 255:red, 0; green, 0; blue, 0 }  ,draw opacity=0.45 ][fill={rgb, 255:red, 248; green, 184; blue, 28 }  ,fill opacity=0.25 ] (311.03,240.54) .. controls (311.04,237.76) and (311.7,235.13) .. (312.87,232.8) -- (328.21,240.61) -- cycle ;
\draw    (346.95,180.03) -- (314,208.03) ;

\draw    (355.95,192.7) -- (313.97,219.95) ;

\draw    (362.29,207.37) -- (314.1,230.18) ;

\draw    (368.29,226.03) -- (317.91,238.5) ;

\draw    (230.86,447.13) -- (251.12,415.48) ;
\draw [shift={(252.19,413.8)}, rotate = 122.62] [color={rgb, 255:red, 0; green, 0; blue, 0 }  ][line width=0.75]    (10.93,-3.29) .. controls (6.95,-1.4) and (3.31,-0.3) .. (0,0) .. controls (3.31,0.3) and (6.95,1.4) .. (10.93,3.29)   ;

\draw (632.4,215) node [anchor=north west][inner sep=0.75pt]   [align=left] {$\mathfrak{Re}$};
\draw (90,6.4) node [anchor=north west][inner sep=0.75pt]   [align=left] {$\mathfrak{Im}$};
\draw (330.21,243.61) node [anchor=north west][inner sep=0.75pt]   [align=left] {$\displaystyle 1+r$};
\draw (127.33,29) node [anchor=north west][inner sep=0.75pt]  [font=\scriptsize] [align=left] {$\displaystyle 1/8$};
\draw (270,9.33) node [anchor=north west][inner sep=0.75pt]   [align=left] {$\displaystyle \sigma =1$};
\draw (346.33,167.71) node [anchor=north west][inner sep=0.75pt]  [font=\footnotesize] [align=left] {$\displaystyle \theta _{1}$};
\draw (357.07,182.51) node [anchor=north west][inner sep=0.75pt]  [font=\footnotesize] [align=left] {$\displaystyle \theta _{2}$};
\draw (362.47,199.51) node [anchor=north west][inner sep=0.75pt]  [font=\footnotesize] [align=left] {$\displaystyle \theta _{3}$};
\draw (371.13,218.11) node [anchor=north west][inner sep=0.75pt]  [font=\footnotesize] [align=left] {$\displaystyle \theta _{4}$};
\draw (140,320) node [anchor=north west][inner sep=0.75pt]   [align=left] {$\displaystyle J_{5}$};
\draw (172,363) node [anchor=north west][inner sep=0.75pt]   [align=left] {$\displaystyle J_{4}$};
\draw (212,399.67) node [anchor=north west][inner sep=0.75pt]   [align=left] {$\displaystyle J_{3}$};
\draw (216,441) node [anchor=north west][inner sep=0.75pt]   [align=left] {$\displaystyle J_{2}$};
\draw (280,432.33) node [anchor=north west][inner sep=0.75pt]   [align=left] {$\displaystyle J_{1}$};
\draw (518,263.86) node [anchor=north west][inner sep=0.75pt]   [align=left] {$\displaystyle J_{0}$};
\draw (150,46) node [anchor=north west][inner sep=0.75pt]  [font=\scriptsize] [align=left] {$\displaystyle 1/4$};
\draw (187.33,29) node [anchor=north west][inner sep=0.75pt]  [font=\scriptsize] [align=left] {$\displaystyle 1/2$};
\draw (228.67,46) node [anchor=north west][inner sep=0.75pt]  [font=\scriptsize] [align=left] {$\displaystyle 3/4$};
\draw (248.67,29) node [anchor=north west][inner sep=0.75pt]  [font=\scriptsize] [align=left] {$\displaystyle 7/8$};

\end{tikzpicture}
\caption{Contour of Integration}
\label{fig:finalcircle}
\end{figure}
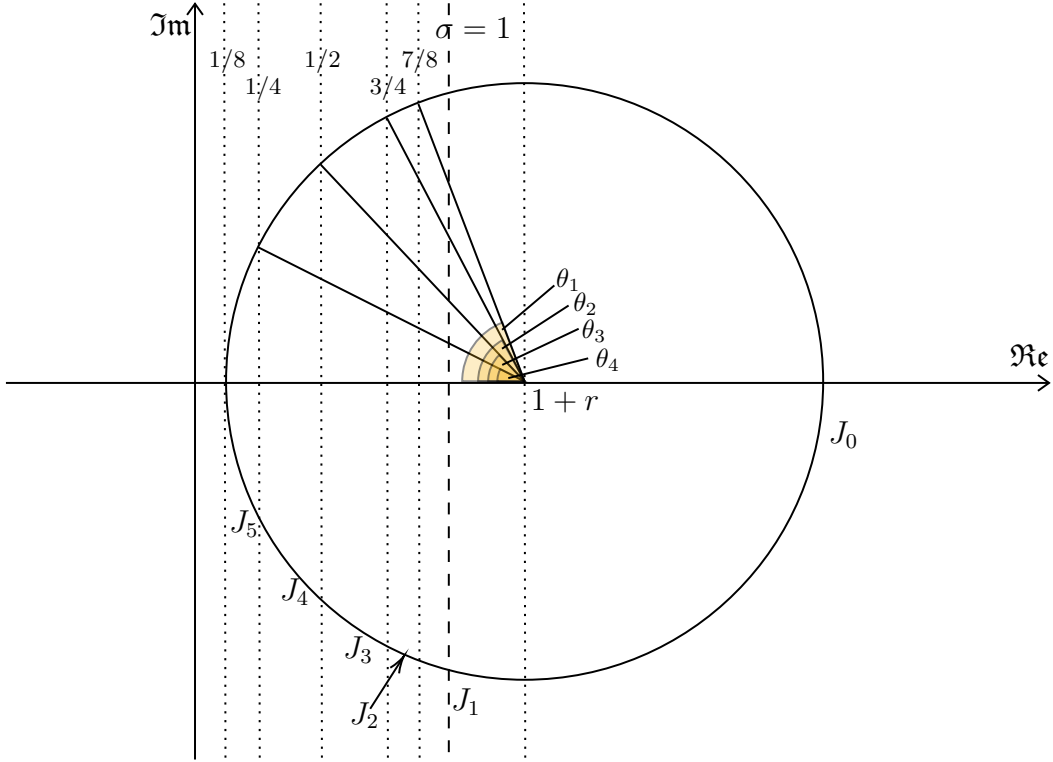

Split up the circle as in \Cref{fig:finalcircle}. More precisely, let $\theta_1=\cos^{-1}((r+1/8)/r_1)$, $\theta_2=\cos^{-1}((r+1/4)/r_1)$, $\theta_3=\cos^{-1}((r+1/2)/r_1)$, $\theta_4=\cos^{-1}((r+3/4)/r_1)$, and let $\theta_0=\pi/2$ and $\theta_5=0$. By conjugate symmetry, it follows that the contribution of $J$ may be divided as follows: 
\begin{equation} \label{eqn:splitJ}
    J=J_0+2J_1+2J_2+2J_3+2J_4+2J_5,
\end{equation}
where 
\[
J_0=\int_{-\pi/2}^{\pi/2}( \dots ) d\theta \quad \text{ and } J_k=\int_{\pi-\theta_{k-1}}^{\pi-\theta_k}(\dots ) d\theta.
\]
For notational convenience, let 
$$A(n,m)=\int_{\pi-\theta_n}^{\pi-\theta_m}\cos\theta d\theta=\sin(\theta_m)-\sin(\theta_n)$$
and
$$B(n,m)=\int_{\pi-\theta_n}^{\pi-\theta_m}\cos^2\theta d\theta=\frac{\theta_n-\theta_m}{2}+\frac{1}{2}(\sin(\theta_n)\cos(\theta_n)-\sin(\theta_m)\cos(\theta_m)).$$

Our key inequality (\Cref{thm:finalbound}) will involve very repetitive calculations, so we will summarize the process in the following lemmas.  The basic idea is to combine the classical Phragm\'{e}n-Lindel\"{o}f principle and \Cref{thm:34bound} between each pair of lines defining $J_1,\dots,J_5$. The first step deals with the strip between $\Re s=1/2$ and $\Re s =7/8$. 

\begin{lemma}
\label[lemma]{thm:rightbound}
    For $k=2,3$, let
    \begin{align*}
    K_k=&\left(2^{6-k}\log C_{2^{4-k}}+2\log(15/8+2^{k-4}+2r)\right)\left(\frac{r+2^{k-5}}{\pi (r+\frac78)}A(k-1,k)+\frac{B(k-1,k)}{\pi}\right)\\
    &-\left(2^{6-k}\log C_{2^{5-k}}+\log(15/8+2^{k-4}+2r)\right)\left(\frac{r+2^{k-4}}{\pi (r+\frac78)}A(k-1,k)+\frac{B(k-1,k)}{\pi}\right).
    \end{align*}
    Then        
    \begin{align*}
    -\lim_{r_1\rightarrow r+7/8}2J_k\le&\left(\frac{rA(k-1,k)}{\pi(r+\frac78)}+\frac{B(k-1,k)}\pi \right)\log q +K_k.
    \end{align*}
\end{lemma}
\begin{proof}
    Note that for $\theta\in[\pi-\theta_{k-1},\pi-\theta_k]$, we have $\Re(1+r+r_1e^{i\theta})\in[1-2^{k-4},1-2^{k-5}]$. Let $Q=2^{k-4}$, then $|L(1-2^{k-4}+it,\chi)|\le C_{2^{4-k}}(q|Q+1-2^{k-4}+it|)^{2^{k-5}}$ and $|L(1-2^{k-5}+it,\chi)|\le C_{2^{5-k}}(q|Q+1-2^{k-5}+it|)^{2^{k-6}}$ by \Cref{thm:34bound}. Now apply the Phragmén–Lindelöf principle as in \cite[Theorem 2]{Rademacher_1959} to get that, for $\theta\in[\pi-\theta_{k-1},\pi-\theta_k]$,
    \begin{align*}
        &|L(1+r+r_1e^{i\theta},\chi)|\\
        \le&\big(C_{2^{4-k}}(q(15/8+2^{k-4}+2r))^{2^{k-5}}\big)^{\frac{-2^{k-5}-r-r_1\cos\theta}{2^{k-5}}}\big(C_{2^{5-k}}(q(15/8+2^{k-4}+2r))^{2^{k-6}}\big)^{\frac{2^{k-4}+r+r_1\cos\theta}{2^{k-5}}},
    \end{align*}
    where we have used that $|Q+1+r+r_1e^{i\theta}|\le 1+2^{k-4}+r+r_1\le\frac{15}{8}+2^{k-4}+2r$. Using monotonicity of $\log$ and evaluating the integral, we find $-J_k$ is bounded by
    \begin{align*}
        &\Big(\log q+2^{5-k}\log C_{2^{4-k}}+\log(\tfrac{15}{8}+2^{k-4}+2r)\Big)\Big(\frac{r+2^{k-5}}{\pi r_1}A(k-1,k)+\frac{B(k-1,k)}{\pi}\Big)\\
    -&\Big(\tfrac{1}{2}\log q+2^{5-k}\log C_{2^{5-k}}+\tfrac{1}{2}\log(\tfrac{15}{8}+2^{k-4}+2r)\Big)\Big(\frac{r+2^{k-4}}{\pi r_1}A(k-1,k)+\frac{B(k-1,k)}{\pi}\Big).
    \end{align*}
    The result follows.
\end{proof}

The second lemma deals with the strip between $\Re s=1/8$ and $\Re s = 1/2$ by using the functional equation.
\begin{lemma}
\label[lemma]{thm:leftbound}
    For $k=4,5$ and $0 < r\le 2^{3-k}$, let
    \begin{align*}
    K_k=&-\Big(2^{k-1}\log C_{2^{k-3}} +2\log(7/8+2^{3-k})\Big)\left(\frac{r+1-2^{2-k}}{\pi(r+\frac78)}A(k-1,k)+\frac{B(k-1,k)}{\pi}\right)\\&+\Big(2^{k-1}\log C_{2^{k-2}}+\log(7/8+2^{3-k})\Big)\left(\frac{r+1-2^{3-k}}{\pi(r+\frac78)}A(k-1,k)+\frac{B(k-1,k)}{\pi}\right).
    \end{align*}
    Then
    \begin{align*}
    -\lim_{r_1\rightarrow r+7/8}&2J_k\le\left(\frac{rA(k-1,k)}{\pi(r+\frac78)}+\frac{B(k-1,k)}\pi \right)\log q +K_k\\
    &+2\Big(\log|2+r-(r+7/8)e^{i\theta_3}|-\log(2\pi)\Big)\left(\frac{r+\frac12}{\pi(r+\frac78)}A(k-1,k)+\frac{B(k-1,k)}{\pi}\right).
    \end{align*}
\end{lemma}
\begin{proof}
    By the functional equation,
    \begin{align*}
    |L(1+r+r_1e^{i\theta},\chi)|&=|L(-r-r_1e^{i\theta},\chi)|\Big(\frac{q}{\pi}\Big)^{-1/2-r-r_1\cos\theta}\left|\frac{\Gamma((-r-r_1e^{i\theta}+\kappa)/2)}{\Gamma((1+r+r_1e^{i\theta}+\kappa)/2)}\right|\\
    &\le|L(-r-r_1e^{i\theta},\chi)|\left(\frac{q}{2\pi}|2+r-r_1e^{i\theta_{3}}|)\right)^{-1/2-r-r_1\cos\theta}
    \end{align*}
    where we have used \cite[Lemma 1 and Lemma 2]{Rademacher_1959} to bound $\Gamma$, and the fact that $|x+r_1e^{i\theta}|\le |x+r_1e^{i(\pi-\theta_{3})}|$ for $\theta\in[\pi-\theta_{k-1},\pi-\theta_k]$ and $x\ge r_1$.
    
    For $\theta\in[\pi-\theta_{k-1},\pi-\theta_k]$, we have $\Re(-r-r_1e^{i\theta})\in[1-2^{3-k},1-2^{2-k}]$. Let $Q=2^{3-k}$, then $|L(1-2^{3-k}+it,\chi)|\le C_{2^{k-3}}(q|Q+1-2^{3-k}+it|)^{2^{2-k}}$ and $|L(1-2^{2-k}+it,\chi)|\le C_{2^{k-2}}(q|Q+1-2^{2-k}+it|)^{2^{1-k}}$ by \Cref{thm:34bound}. By Phragmén–Lindelöf, for $\theta\in[\pi-\theta_{k-1},\pi-\theta_k]$, we have that
    \begin{align*}
        &|L(-r-r_1e^{i\theta},\chi)|
        \\\le&\Big(C_{2^{k-3}}(q(7/8+2^{3-k}))^{2^{2-k}}\Big)^{\frac{1-2^{2-k}+r+r_1\cos\theta}{2^{2-k}}}\Big(C_{2^{k-2}}(q(7/8+2^{3-k}))^{2^{1-k}}\Big)^{\frac{2^{3-k}-1-r-r_1\cos\theta}{2^{2-k}}}
    \end{align*}
    where we have used that $|Q-r-r_1e^{i\theta}|\le 2^{3-k}-r+r_1\le 7/8+2^{3-k}$. Evaluating the integral, we find that $-J_k$ is bounded by
    \begin{align*}
        &\Big(\log q+\log|2+r-r_1e^{i\theta_{3}}|-\log(2\pi)\Big)\left(\frac{r+1/2}{\pi r_1}A(k-1,k)+\frac{B(k-1,k)}{\pi}\right)\\
    -&\Big(\log q+2^{k-2}\log C_{2^{k-3}}+\log(\tfrac78+2^{3-k})\Big)\left(\frac{r+1-2^{2-k}}{\pi r_1}A(k-1,k)+\frac{B(k-1,k)}{\pi}\right)\\
    +&\Big(\tfrac12\log q+2^{k-2}\log C_{2^{k-2}}+\tfrac12\log(\tfrac78+2^{3-k})\Big)\left(\frac{r+1-2^{3-k}}{\pi r_1}A(k-1,k)+\frac{B(k-1,k)}{\pi}\right).
    \end{align*}
    The result follows.
\end{proof}

\subsection{Final estimates} We assemble our components to establish our main technical tools.

\begin{theorem} \label[theorem]{thm:finalbound}
Let  $\chi$ be a primitive character modulo $q \ge 3$. For $r \in (0,1/4)$, let
$$\phi=\left(\frac{rA(1,5)}{\pi(r+\frac78)}+\frac{B(0,1)}{\pi(8r+1)}+\frac{B(1,5)}{\pi}\right),$$
$$\rho=-\left(\frac{2A(0,1)}{\pi (r+\frac78)}+\frac{2B(0,1)}{\pi(r+\frac18)}\right),$$
$$K_1=\Big(2\log C_8+\frac{1}{8}\log(2r+2)\Big)\frac{B(0,1)}{\pi(r+\frac18)},$$
$$K_R=2\Big(\log|2+r-(r+7/8)e^{i\theta_3}|-\log(2\pi)\Big)\left(\frac{r+\frac12}{\pi(r+\frac78)}A(3,5)+\frac{B(3,5)}{\pi}\right),$$
and $K_i$ for $2\le i\le5$ be as in Lemmas \ref{thm:rightbound} and \ref{thm:leftbound}. If $L(s,\chi)$ has a real zero $\beta_1 \in (0,1)$, then
    \begin{align*}
        -\Re \ld{L}(1+r,\chi)&\le\frac{1+r-\beta_1}{(r+\frac78)^2}-\frac{1}{1+r-\beta_1}+\phi\log q+\rho\log(1+r^{-1})\\
        &\quad+\frac{1.912}{\pi(r+\frac78)}+K_R+\sum_{j=1}^5K_j.
    \end{align*}
\end{theorem}
\begin{proof}
Apply \Cref{thm:boundJ}, then it remains to bound $J$. When $r_1\rightarrow r+7/8$, we have that 
\begin{align*}|J_0|\rightarrow\left|\frac{1}{\pi (r+7/8)}\int_{-\pi/2}^{\pi/2}(\cos\theta)\log|L(1+r+(r+7/8)e^{i\theta},\chi)|d\theta\right|.
\end{align*}

By \Cref{thm:J0bound}, the right-hand quantity is at most 
\begin{align*}
    &\le \frac{1}{\pi (r+7/8)}\left(\tfrac{8}{7}\pi-4\sqrt{\tfrac{15}{49}}\tan^{-1}\left(\tfrac{1}{\sqrt{15}}\right)-2\log(\tfrac{7}{4})\right)\le\frac{1.912}{\pi(r+7/8)}.
\end{align*}

For $J_1$, when $\theta\in[\pi/2,\pi-\theta_1]$, apply Phragmén–Lindelöf along $\Re s=7/8$ using \Cref{thm:34bound} and along $\Re s =1+r$ using the trivial bound $|L(s,\chi)| \le \zeta(\Re s)\le 1+r^{-1}$ to obtain
\begin{align*}
    |L(1+r+r_1e^{i\theta},\chi)|\le (C_8(q(2+2r))^{1/16})^{\frac{-r_1\cos\theta}{r+1/8}}(1+r^{-1})^\frac{1/8+r+r_1\cos\theta}{r+1/8},
\end{align*}
where we have used that $|9/8+r+r_1e^{i\theta}|\le9/8+r+r_1\le 2+2r$.
Hence,
\begin{align*}
    -2J_1&\le\frac{B(0,1)}{\pi(8r+1)}\log q+K_1-\log(1+r^{-1})\left(\frac{2A(0,1)}{\pi r_1}+\frac{2B(0,1)}{\pi(r+1/8)}\right).
\end{align*}
The third term here becomes $\rho\log(1+r^{-1})$ as $r_1\rightarrow r+7/8$.

Apply \Cref{thm:rightbound} to $J_2$ and $J_3$, and apply \Cref{thm:leftbound} to $J_4$ and $J_5$. Collecting our estimates into $-J\le |J_0|-2J_1-2J_2-2J_3-2J_4-2J_5$ by \eqref{eqn:splitJ}, we obtain our result.
\end{proof}

\begin{remark} The choice $r+7/8$ corresponds to the largest circle we can use without passing into the region $-r<\Re s<1/8$. The bound for this region is analogous to our $J_1$ bound of the region $r+7/8<\Re s<1+r$, and results in a $\log(1+r^{-1})$ term. However, $-\cos\theta$ is maximized in the leftmost region so the coefficients are larger. Since $\log(1+r^{-1})$ is relatively large, avoiding a large coefficient on this term is preferable.

The value of 7/8 here of course comes from our situational choice to use the $\Re s=7/8$ bound in \Cref{thm:34bound} but not the $\Re s=15/16$ bound.
\end{remark} 
\begin{remark}
 \Cref{thm:finalbound} is tailored specifically for our size of moduli (i.e. $q$ near $10^{10}$). We can actually get $\phi=\frac{1}{4+8r}$ by applying the much simpler construction in \cite[Proposition 2.11]{Thorner_Zaman_2024}, which is smaller than our $\phi$ value in \Cref{thm:finalbound} and hence yields a better bound for fixed $r$ and $q\rightarrow\infty$. However, the simpler construction yields larger lower-order terms which are relevant for the size of moduli under consideration.
\end{remark}

Combining the above with \Cref{thm:zetabound}, we get our desired inequality of type \eqref{eqn:ineq}. 

\begin{corollary}
\label[corollary]{thm:combineZetaL}
     Keep the notation and assumptions of \Cref{thm:finalbound}.  
    Let 
    $$K=\frac{1.912}{\pi (r+7/8)}+K_R+\sum_{j=1}^5K_j-\log 2\pi+\frac{\gamma}{2}+1+\frac{1}{2}\ld{\Gamma}\left(\frac{3+r}{2}\right)$$
    and let $E=K+\rho\log(1+r^{-1})$. If $L(s,\chi)$ has a real zero $\beta_1 \in (0,1)$ then 
    \begin{equation}
        -\Re\ld{\zeta}(1+r)-\Re\ld{L}(1+r,\chi)\le \frac{1}{r}-\frac{1}{1+r-\beta_1}+\frac{1+r-\beta_1}{(r+7/8)^2}+\phi\log q+E.
        \label{eq:combineZetaLineq}
    \end{equation}
\end{corollary}

\begin{remark}
    \label{rem:suboptimalBounds}
    One can show that $\phi(r)$ is decreasing in $r$ and that $\phi(0)=1/4$. Moreover, $\rho$ and $K$ are also $O(1)$ which implies that $E=O(\log(1+r^{-1}))$. In fact, based on numerical experimentation, it appears that $\rho\le 0.023$ uniformly and that $K$ is decreasing, which would imply $K(r)\le K(0)\le 1.26$. We did not attempt to rigorously verify these claims since we will only need to evaluate the expressions at finitely many numerical values of $r$. 
\end{remark}

\section{Proof of Theorem 2.1 and choosing parameters}
\label{sect:optimizeR}

\begin{proof}[Proof of \Cref{thm:explicitvalues}]
Substituting $r = 1.6/10\log10$, $1.3/10\log10$, $1.2/10\log10$, and 

\noindent$1.1/10\log10$ into \Cref{thm:combineZetaL}, we therefore obtain the constants in \Cref{table:ineq}. The code to perform these calculations can be found in our GitHub at \href{https://github.com/asif-z/Landau-Siegel-Zero-Tester/blob/master/misc/rigorousBound.c}{\tt misc/rigorousBound.c} (cf. \href{https://github.com/asif-z/Landau-Siegel-Zero-Tester/blob/master/misc/ComputePhiAndE.md}{\tt misc/ComputePhiAndE.md} for the same computation in Wolfram Language).
\end{proof}

As explained in \S\ref{sect:proof-main}, we fix $\sigma = 1+r=1+\lambda/10\log 10$ since we want to optimize our algorithm for moduli $q \approx 10^{10}$. The choice of our values $\lambda$ was somewhat ad hoc. The basic idea is as follows. Let $S$ denote the right side of \eqref{eq:combineZetaLineq}. For a fixed choice of $\lambda$ and $N$,  our inequality from \Cref{thm:combineZetaL} combined with \eqref{eqn:non-negativity} and \eqref{eqn:primesum} reads as $\sum_{p\le N} \left( \frac{\log p}{p^{\sigma}-1}  + \frac{ \chi(p)  \log p}{p^{\sigma}- \chi(p) } \right)\le S$. By minimizing the quantity $\Delta_N(\lambda) := S - \sum_{p \leq N} \frac{\log p}{p^{\sigma}-1}$, we can presumably increase the likelihood for which  a given $\chi$ will violate this inequality. After some numerical experimentation, we found that the values $\lambda=1.6$, $1.3$, $1.2$, and $1.1$ in \Cref{table:ineq} are approximately the values which minimize $\Delta_N(\lambda)$ when $N$ is the $60,000^\textnormal{th}$, $5,000,000^\textnormal{th}$, $50,000,000^\textnormal{th}$, and $130,000,000^\textnormal{th}$ prime respectively.

\section*{Acknowledgements}
\noindent
We are grateful to Dave Platt, Jesse Thorner, and Tim Trudgian for their comments and encouragement on an earlier version of this manuscript. We also wish to thank the anonymous referees for their detailed reading and helpful suggestions on the exposition. This project was partially funded by NSERC and UTEA, and organized by the Department of Mathematics at the University of Toronto as part of the 2025 Summer Undergraduate Student Research Awards. We are thankful for their generosity, and for the support of Compute Canada (\href{https://alliancecan.ca/en}{alliancecan.ca}) where most of our computation was performed. AZ was partially supported by NSERC grant RGPIN-2022-04982. AZ prepared part of this work while visiting Kyushu University, and is thankful for their warm hospitality and excellent working conditions.

\bibliographystyle{alpha}
\bibliography{references}
\parindent0pt

\end{document}